\documentclass[12pt,reqno]{amsart}  
\usepackage[usenames]{color}
\usepackage{amsmath,amsthm,amsfonts,amssymb,stmaryrd} 

\numberwithin{equation}{section} 
\newtheorem{thm}[equation]{Theorem}
\newtheorem{cor}[equation]{Corollary} 
\newtheorem{lemma}[equation]{Lemma}

\theoremstyle{definition} 
\newtheorem{defn}[equation]{Definition} 
\newtheorem{remark}[equation]{Remark}

\newcommand{\rfac}[2]{{\left[#1\right]}^{\![#2]}}

\newcommand{\edge}[3]{{}^{\,#2}_{#1 ,\, #3}}
\newcommand{\eedge}[1]{{}^{#1}}
\newcommand{\e}[1]{a#1}

\newcommand{\kk}{{\sf k}}

\newcommand{\ot}{\otimes} 

\newcommand{\Z}{{\mathbb Z}}

\begin{document}

\begin{abstract}
We apply a combinatorial formula of the first author and Rosso, 
for products in Hopf quiver algebras,
to determine the structure of Nichols algebras. 
We illustrate this technique 
by explicitly constructing new examples of Nichols algebras in 
positive characteristic.
We further describe the 
corresponding Radford biproducts and some liftings of these biproducts,
which are new finite dimensional pointed Hopf algebras.
\end{abstract}

\title[Hopf quivers]
{{H}opf quivers and {N}ichols algebras in positive characteristic}

\author{Claude Cibils}
\address{Institut de Math\'ematiques et de Mod\'elisation de Montpellier,
Universit\'e Montpellier 2, F-34095 Montpellier Cedex 5, France}
\email{Claude.Cibils@math.univ-montp2.fr}

\author{Aaron Lauve}
\address{Department of Mathematics, Texas A\&M University,
College Station, TX 77843-3368, USA}
\email{lauve@math.tamu.edu}
\thanks{The second and third authors were partially supported by the Texas Advanced Research Program Grant \#010366-0046-2007.}

\author{Sarah Witherspoon}
\address{Department of Mathematics, Texas A\&M University,
College Station, TX 77843-3368, USA}
\email{sjw@math.tamu.edu}
\thanks{The third author was partially supported by NSA grant H98230-07-1-0038
and NSF grant DMS-0800832.}

\date{March 31, 2009}

\maketitle

\section{Introduction}

Nichols \cite{Ni} anticipated the study of pointed Hopf algebras in 
defining what he termed bialgebras of rank one.
Since then many mathematicians have worked to understand what are now
called Nichols algebras and related pointed Hopf algebras, which 
include quantum universal enveloping algebras and their finite quotients
at roots of unity.
In this paper, we study Nichols algebras via an embedding in Hopf quiver algebras.
We apply a combinatorial formula of the first author and Rosso \cite{CR} for products
in Hopf quiver algebras to obtain relations in Nichols algebras.
As a result, we discover some new examples in positive characteristic.
This formula (see \eqref{CR formula} below) was used by van Oystaeyen and Zhang \cite{vOZ} 
to classify simple-pointed Hopf subalgebras of Schurian Hopf quiver algebras
and by Zhang, Zhang, and Chen \cite{ZZC} to understand Hopf quiver algebras for 
certain types of quivers.
Our main results show that the formula is useful and computable in more
general settings, and thus is a valuable tool in the study of the 
fundamental Nichols algebras. 

Quivers and Hopf algebras first appeared together in the work of Green \cite{G} and Green and Solberg \cite{GS}. They studied finite dimensional basic Hopf algebras by developing a theory of path algebras having Hopf algebra structures.
The first author and Rosso \cite{CR97} also worked on path algebras, and then developed 
their dual theory of Hopf quiver algebras \cite{CR}. These are path coalgebras having pointed Hopf algebra structures. 
Van Oystaeyen and Zhang \cite{vOZ} use the theory of Hopf quivers to prove a dual Gabriel Theorem for pointed Hopf algebras.

Each pointed Hopf algebra
has coradical a group algebra. When the group is abelian and the field
has characteristic 0, Andruskiewitsch and Schneider \cite{AS} largely
classified the finite dimensional pointed Hopf algebras
as liftings
of Radford biproducts (or bosonizations) of Nichols algebras.
Only a few examples are known when the group is nonabelian, however
there is much recent progress 
(see \cite{AF08,AZ07,HS} and references given in these papers).
Less has been done in positive characteristic.
Scherotzke \cite{SSch} classified finite dimensional
pointed rank one Hopf algebras in positive characteristic
that are generated by grouplike and skew-primitive elements.
Her work shows that there are more possibilities
than occur in characteristic 0 (cf.\ Krop and Radford \cite{KR}).

In Section \ref{preliminaries} we collect basic facts about Hopf quivers
and Nichols algebras, including the product formula of the first author and Rosso (Theorem \ref{product-formula} below).
A product of paths in the Hopf quiver algebra is a certain sum of paths. Relations result from cancellations of paths, providing a combinatorial visualization of
the algebra structure.

In Section \ref{charp}, we use the product formula to describe Nichols algebras and pointed Hopf algebras in
positive characteristic $p$ arising from certain two-dimensional 
indecomposable  Yetter--Drinfeld modules
of cyclic groups. 
These modules are nonsemisimple, and so 
are purely a positive characteristic phenomenon. 
If the cyclic group $G$ has order~$n$ (divisible by $p$), these modules
lead to Hopf algebras of dimension $16n$ in characteristic 2
(Theorem~\ref{rank2-char2}) 
and of dimension $p^2n$ in odd characteristic $p$ (Theorem \ref{rank2-charp}).
These Hopf algebras all have rank two in the sense that the dimension of
the space of skew-primitive elements, modulo the group algebra of $G$, is two.

The Hopf algebras of Section \ref{charp} are (coradically) graded.
More general pointed Hopf algebras in
positive characteristic may be found by applying the lifting method
of Andruskiewitsch and Schneider \cite{AS} to obtain filtered pointed Hopf algebras
whose associated graded Hopf algebras are those of Section \ref{charp}.
We give some examples of such liftings in Section \ref{future}.

\section{Hopf quivers and Nichols algebras}\label{preliminaries}

Let $\kk$ be a field and $G$ a finite group.
Our notation will be taken from \cite{AS02} and \cite{CR}. 

A (right) {\em Yetter--Drinfeld} module $V$ over $\kk G$ is a $G$-graded
$\kk G$-module for which $({}^gV)\cdot h = {}^{h^{-1}gh}V$ for all $g,h\in G$
(left superscripts denote the grading).
We will use the following standard construction of such modules.
(See \cite[Section 3.1]{AG} in the semisimple setting.
The ``same'' construction works in the nonsemisimple setting,
for example, see \cite[Prop.\ 3.3 and Rem.\ 3.5]{CR97}.)

For fixed $g\in G$, let $Z(g)$ denote its centralizer and let $M$ be a right $\kk Z(g)$-module. Take for $V$ the right $\kk G$-module induced from $Z(g)$, that is,
\[
    V := M\ot_{\kk Z(g)} \kk G,
\]
where $\kk G$ acts on the right tensor factor by multiplication.
Then $V$ becomes a Yetter--Drinfeld module over $\kk G$ by setting ${}^{k^{-1}gk}V:=M\ot k$ 
for all $k\in G$, and ${}^lV:=0$ if $l$ is not conjugate to $g$. Let
$$
 B=B(M) := \kk G \ot_{\kk} M \ot_{\kk Z(g)} \kk G
$$
be the corresponding {\em Hopf bimodule} over $\kk G$.
That is, $B$ has a $G$-bigrading (equivalently is a $\kk G$-bicomodule)
given by
$$
   {}^{h k^{-1}gk} B ^h:= h\ot M\ot k,
$$
and $B$ is a $\kk G$-bimodule on which $\kk G$ acts by left and right multiplication.
Following~\cite{CR}, we realize $B$ in terms of a {\em Hopf quiver:} this is a directed multi-graph $Q$ with vertex set $G$ and arrows (directed edges) determined as follows. There are arrows from $h$ to $h'$ if and only if $h'=hk^{-1}gk$ for some $k\in G$; these arrows are in one-to-one correspondence with a basis of the vector space $h\ot M\ot k$. 
Fixing a basis $\{ m_i : i \in \mathcal I\}$ for $M$, and a set
of representatives $k$ of cosets $Z(g)\backslash G$, we 
denote such an arrow by 
\[
   \e{\edge{h'}{i}{h}}= \e{\edge{h'}{m_i}{h}} := h\ot m_i\ot k
\]
if $h'=hk^{-1}gk$. 
The {\em source} and {\em target} for this arrow are 
$s(\e{\edge{h'}{i}{h}})=h$ and $t(\e{\edge{h'}{i}{h}})=h'$, respectively. 
When convenient, we will also write, for any $m\in M$,
\[
   \e{\edge{hk^{-1}gk}{m}{h}} := h\ot m\ot k. \]

Finally, we let $\kk Q$ denote the {\em path coalgebra} over $Q$. The underlying vector space for $\kk Q$ has basis all {\em paths} in $Q$, i.e., sequences of arrows $a_n\cdots a_2a_1$ such that $s(a_{j+1})=t(a_j)$ for all $1\leq j\leq n$. (In the interest of clarity, this is often written as $[a_n,\ldots ,a_2,a_1]$ in what follows.) Paths of length zero (vertices) are also allowed. 
Coproducts in $\kk Q$ are determined by $\Delta(h)=h\ot h$ for all $h\in G$ and 
\[
\Delta(a_n\cdots a_2a_1) = \sum_{i=0}^{n} a_n\cdots a_{i+1}\ot a_i \cdots a_1, 
\]
where $a_0$ and $a_{n+1}$ are understood as $s(a_1)$ and $t(a_n)$ respectively. Note that $\e{\edge{h}{m}{1}}$ is skew-primitive for all $h\in G$ and $m\in M$. The path coalgebra $\kk Q$ is also known as the {\em cotensor coalgebra}
of $B$ over $\kk G$ (see the text preceding \cite[Defn.\ 2.2]{CR}). Theorem 3.3 of \cite{CR} provides $\kk Q$ with an algebra structure that makes it a pointed graded Hopf algebra. The algebra structure is determined by the $\kk G$-bimodule structure of $B$, which we now describe explicitly, following \cite{CS}. 

The left action of $\kk G$ on $B$ is diagonal, with the left regular
action on the left factor of $\kk G$ and the trivial action on the induced module
$M\ot_{\kk Z(g)} \kk G$. The right action of $\kk G$ on $B$ is diagonal,
with the right regular action on the left factor $\kk G$ and the induced action
on $M\ot_{\kk Z(g)} \kk G$. That is,
\begin{align}
  l\cdot \e{\edge{hk^{-1}gk}{m}{h}} &\ =\  \e{\edge{lhk^{-1}gk}{m}{lh}},
\hbox{ and} 
\label{laction}\\
  \e{\edge{hk^{-1}gk}{m}{h}}\cdot l &\ =\  \e{\edge{hl(k')^{-1}gk'}{m\cdot l'}{hl}}
 \ = \ \e{\edge{hk^{-1}gkl}{m\cdot l'}{hl}}, \label{raction}
\end{align}
where $k'$ is the unique representative of $Z(g)\backslash G$ such that $kl=l'k'\in Z(g)k'$.
The formula for multiplying two arrows from \cite[\S3]{CR} is
$$
  a\cdot b = [t(a)\cdot b][a\cdot s(b)] + [a\cdot t(b)][s(a)\cdot b],
$$
where the square brackets $[\ldots ][\ldots ]$ mean ``concatenate paths.'' In particular,
\begin{align*}
\notag  \e{\edge{k^{-1}gk}{m}{1}}\cdot \e{\edge{h^{-1}gh}{n}{1}} &=
    [{ k^{-1}gk}\cdot \e{\edge{h^{-1}gh}{n}{1}}][\e{\edge{k^{-1}gk}{m}{1}}\cdot { 1}]
  + [\e{\edge{k^{-1}gk}{m}{1}}\cdot { h^{-1}gh}][{ 1}\cdot \e{\edge{h^{-1}gh}{n}{1}}] \\
\notag &= [\e{\edge{k^{-1}gkh^{-1}gh}{n}{k^{-1}gk}}][\e{\edge{k^{-1}gk}{m\cdot1}{1}}]
   +[\e{\edge{k^{-1}gkh^{-1}gh}{m\cdot l'}{h^{-1}gh}}]
   [\e{\edge{h^{-1}gh}{n}{1}}]  
, 
\end{align*}
where $kh^{-1}gh=l'k'$ as above. This is
a sum  of paths of length 2 starting at 1 and ending at
$k^{-1}gkh^{-1}gh$.

\begin{remark}
More generally, one may fix several group elements $g_1,\ldots, g_s$ and let $V$ be a Yetter--Drinfeld module over $\kk G$ of the form $V=\oplus_{j=1}^s (M_j\ot_{\kk Z(g_j)} \kk G)$.
The corresponding Hopf bimodule and Hopf quiver are defined analogously.
While \cite[Thm.~3.3]{CR} holds in this greater generality, we restrict our attention to the case $s=1$ here. 
\end{remark}

Evidently, the above algebra structure is not the one commonly placed on $\kk Q$; we call that one the ``path algebra product'' or ``concatenation product'' here. The path algebra structure will be used in what follows, so we recall the details now. It is the $\kk$-algebra with vertices and arrows as generators subject to the following relations: vertices $v$ are orthogonal idempotents, and arrows $a,a' \in Q$ satisfy $av=\delta_{s(a),v}a$, $va'=\delta_{v,t(a')}a'$ and $aa'= 0$ whenever $s(a)\neq t(a')$. 

Returning to the path coalgebra $\kk Q$,
to describe products for paths of length $n>1$, we follow \cite{CR} and introduce the notion of $p{\scriptscriptstyle\,}$-{\em thin splits}. A $p{\scriptscriptstyle\,}$-{thin split} of a path $\alpha$ in $Q$ is a sequence $(\alpha_p, \ldots, \alpha_2,\alpha_1)$ of vertices and arrows such that the path algebra product $\alpha_p\cdots\alpha_{2} \alpha_1$ is $\alpha$. 
The $p{\scriptscriptstyle\,}$-thin splits of a path $\alpha=a_n\cdots a_2a_1$ are in bijection with the set $D^p_n$ of all cardinality-$n$ subsets of $\{1,2,\ldots,p\}$ (the choice of where to place the $n$ arrows $a_i$ uniquely determines where and which vertices appear in the sequence). We view $d\in D^p_n$ as a function $\alpha \mapsto d\alpha=((d\alpha)_p,\ldots,(d\alpha)_2,(d\alpha)_1)$ from paths to $p{\scriptscriptstyle\,}$-thin splits and write $\overline{d}$ for $\{1,2,\ldots,p\} \setminus d$ (viewed as $p{\scriptscriptstyle\,}$-thin split instructions for paths of length $p-n$). Finally, given paths $\alpha$ and $\beta$ of lengths $n$ and $m$, respectively, and a choice $d\in D^{n+m}_n$, we define a product on the $(n+m)$-thin split of $\alpha$ and $\beta$ along $d$ by the formula
\begin{equation*}\label{thin-split product}
(d\alpha)(\overline d\beta) : =  [(d\alpha)_{m+n}\cdot (\overline{d}\beta)_{m+n}]
  \cdots [(d\alpha)_2\cdot (\overline{d}\beta)_2] [(d\alpha)_1\cdot (\overline{d}\beta)_1],
\end{equation*}
where each $(d\alpha)_i \cdot (\overline d\beta)_i$ is an instance of either \eqref{laction} or \eqref{raction} and square brackets are again understood as the concatenation product. 
The following result is due to the first author and Rosso.

\begin{thm}[{\cite[Thm.\ 3.8]{CR}}]
\label{product-formula}
Let $\alpha$ and $\beta$ be paths of lengths $n\geq0$ and $m\geq0$, respectively,
in $Q$.
Their product in the Hopf quiver algebra $\kk Q$ is given by
\begin{gather}\label{CR formula}
  \alpha\cdot\beta = \sum_{d\in D^{n+m}_n} (d\alpha)(\overline{d}\beta).
\end{gather}
\end{thm}

If $n=m=0$ above, i.e., $\alpha$ and $\beta$ are vertices, then $d=\overline d =\emptyset$ and \eqref{CR formula} reduces to the ordinary product in $G$. This completes the description of the algebra structure on $\kk Q$. We are ready to define Nichols algebras in our context, following \cite[\S2.2]{Ni}.

\begin{defn} \label{defn:narb}
Fix a Yetter--Drinfeld module $V$ over $\kk G$.
The {\em Nichols algebra} ${\mathcal B}(V)$ is the subalgebra of $\kk Q$ 
generated by $V$ (equivalently, by the arrows $\e{\edge{k^{-1}gk}{m}{1}}$).
Its {\em Radford biproduct} (or {\em bosonization})
is the subalgebra of $\kk Q$ generated by $V$ and $G$.
\end{defn}

The Radford biproduct is often denoted 
${\mathcal B}(V)\# \kk G$, using {\em smash product} notation, and this notation will also be used elsewhere, so we define it now (see e.g.\ \cite{Mo}):
If $A$ is any $\kk$-algebra on which the group $G$ acts on the right by automorphisms,
the smash product algebra $A\# \kk G$ is $\kk G\ot A$ as a vector space, and multiplication
is given by $(g\ot a)(h\ot b):= gh\ot (a\cdot h)b$ for all $a,b\in A$ and
$g,h\in G$. We abbreviate the element $g\ot a$ by $ga$.
Sometimes the smash product algebra has the additional structure of a Hopf algebra,
as is the case with the Radford biproduct.

\begin{remark}\label{skew-prim}
{\em (i)\ }
There are many equivalent definitions of  Nichols algebras and of 
the corresponding Radford biproducts.
See in particular the elaboration by Andruskiewitsch and Gra\~na
in \cite[\S3.2]{AG} of Nichols' original definition: The Radford
biproduct  ${\mathcal B}(V)\# \kk G$ is the image of the unique bialgebra
map from the path algebra to the path coalgebra 
preserving $Q$.
Other useful definitions are summarized in \cite[\S\S2.1,2.2]{AS00}, including
Schauenburg's elegant formulation of Woronowicz' quantum symmetrizers
(see \cite{Sch96,Wo}).

{\em (ii)\ }
A potential advantage of the subalgebra formulation $\mathcal B(V) \subseteq \kk Q$ is that one may be able to write relations more simply or find them more readily, since $\kk Q$ is a Hopf algebra with a combinatorial basis whose structure is governed by the combinatorial formula \eqref{CR formula}. Also, \eqref{CR formula} is particularly well-suited to computation; the new algebras we present in Section \ref{charp} were found using an implementation in {\em Maple} \cite{Maple}.
\end{remark}

The Hopf quiver algebra $\kk Q$ is graded by the {\em nonnegative} integers, 
where we assign degree 0 to vertices and degree 1 to arrows. 
As a consequence of the 
definitions and the coalgebra structure of $\kk Q$, 
the skew-primitive elements in the Radford biproduct 
${\mathcal B}(V)\# \kk G$ all lie in degrees 0 and 1. Conversely, we have the 
following theorem, which
is an immediate consequence of \cite[Defn.\ 2.1 and Prop.\ 2.2]{AS02}
(valid in arbitrary characteristic)
upon translation into a statement about smash products.
See also \cite{AG}.
Let $T(V)$ denote the tensor algebra of $V$.
The smash product $T(V)\# \kk G$ is a Hopf algebra with $\Delta(g)=g\ot g$
for all $g\in G$ and $\Delta(v)=v\ot 1+g\ot v$ for all $v\in V_g$.

\begin{thm}[Andruskiewitsch--Schneider]\label{check-skew-prims}
Let $\tilde{I}$ be a homogeneous Hopf ideal of the smash product $T(V)\# \kk G$, generated by elements of $T(V)$ in degrees $\geq2$. The quotient $T(V)\# \kk G / \tilde{I}$ is isomorphic to ${\mathcal B}(V)\# \kk G$ if the skew-primitive elements in $T(V)\# \kk G / \tilde{I}$ all lie in degrees 0 and 1.
\end{thm}

We will use this theorem in  Section \ref{charp} to find new Nichols algebras as follows. 
We first apply the formula \eqref{CR formula} to find relations among
the generators of the subalgebra $\mathcal B(V)$ of $\kk Q$. 
We next take the quotient of $T(V)\# \kk G$ by these relations and show that the quotient has skew-primitives only in degrees $0$ and $1$. Theorem \ref{check-skew-prims} then ensures that we have found defining relations for $\mathcal B(V)$.

\section{New Nichols algebras in positive characteristic}\label{charp}

In this section we use the Hopf quiver approach to construct new Nichols algebras and related finite dimensional pointed Hopf algebras in positive characteristic.

Throughout this section, 
let $\kk$ be a field of characteristic $p$ and let $G=\langle g \rangle$ be the cyclic group $\Z/n\Z$, with $n$ divisible by $p$. 
Let $V$ be the following indecomposable (nonsemisimple) $\kk G$-module of dimension 2:
$V$ has  basis $\{v_1,v_2\}$ for which the (right) action of $\kk G$
is defined by
$$
  v_1 \cdot g = v_1, \ \ \ v_2 \cdot g = v_1+v_2.
$$
Since $G$ is abelian, $V$ is automatically a Yetter--Drinfeld module
with $V_g=V$, and 
the corresponding Hopf bimodule is $\kk G\ot V$. Let
$$ 
    \e{\edge{g^{j+1}}{i}{g^j}} = g^j\ot v_i \ \ \
     (1\leq i\leq 2, \ 0\leq j\leq n-1),
$$
and let
$$
   a=\e{\edge{g}{1}{1}} = 1\ot v_1 \ \mbox{ and } \ b=\e{\edge{g}{2}{1}}
        =1\ot v_2.
$$
Let $\kk Q$ be the corresponding Hopf quiver algebra.
We wish to determine the structure of the subalgebra of $\kk Q$
 generated by $a$ and $b$; by Definition \ref{defn:narb} 
this is the Nichols algebra ${\mathcal B}(V)$. 
This structure is somewhat different when $p=2$, so we deal with this case first.

\begin{thm}\label{rank2-char2}
Assume $\kk$ has characteristic $2$.
The Nichols algebra ${\mathcal B}(V)$ is the quotient of $T(V)$
by the ideal generated by
$$
   a^2, \quad\,\, b^4, \quad\,\, b^2a+ab^2+aba, \quad\,\, abab+baba.
$$
It has dimension $16$ and basis
$$
 1, \ a, \ b, \ ab, \ ba, \ b^2, \ aba, \ ab^2, \ bab, \ b^3, \ abab, \ ab^3, \ bab^2,
    \ abab^2, \ bab^3, \ abab^3.
$$
\end{thm}

\begin{remark}\label{change-of-basis}
Making the change of basis $(v_1,v_2) \mapsto (v_1+v_2,v_2)$ in $V$ yields symmetric relations on new generators 
$\{c ,d \}$. The new relations are
\[
c ^2+c  d  +d  c   +d ^2 = 0, 
\,\quad c ^2d  + c d ^2 + d ^2c  + d c ^2 = 0, 
\,\quad c  ^4=d  ^4=0, 
\,\quad c  ^2d  ^2 + d  ^2c  ^2 = 0.
\]
We leave the details to the reader.
\end{remark}

Let  $\mathcal A(V)$ be the quotient of $T(V)=\kk\langle a,b\rangle$ by the ideal 
\[
  I=\langle a^2, \ b^4,\  b^2a+ab^2+aba,\  abab+baba\rangle.
\] 
We omit the straightforward check that ${\mathcal A}(V)$ is finite dimensional with basis as stated in the theorem. Our proof that $\mathcal A(V)=\mathcal B(V)$ proceeds in three parts:
\begin{enumerate}
\item[(i)] the Nichols algebra $\mathcal B(V)$ has the claimed relations;

\item[(ii)] the ideal $\tilde{I}=\langle a^2$, $b^4$, $b^2a+ab^2+aba$, $abab+baba \rangle$ of $T(V)\# \kk G$ is a Hopf ideal;

\item[(iii)] the quotient of $T(V)\# \kk G$ by $\tilde{I}$ has skew-primitives only in degrees $0$ and $1$.
\end{enumerate}

\begin{proof}[Proof of (i)] $\mathcal B(V)$ is generated by $a=\e{\edge{g}{1}{1}}$ and $b=\e{\edge{g}{2}{1}}$. Theorem \ref{product-formula} gives
\begin{eqnarray*}
    a^2\ = \ {[\e{\edge{g}{1}{1}}]}\cdot [\e{\edge{g}{1}{1}}] &=& 2[\e{\edge{g^2}{1}{g}},\e{\edge{g}{1}{1}}] \ = \ 0, \\
   b^2 \ = \ {[\e{\edge{g}{2}{1}}]}\cdot [\e{\edge{g}{2}{1}}] & = & 
       {(\e{\edge{g^2}{2}{g}},1)}{(g,\e{\edge{g}{2}{1}})} + 
                  {(g,\e{\edge{g^2}{2}{g}})(\e{\edge{g}{2}{1}} , 1)} \\
                  & = & {[\e{\edge{g^2}{2}{g}}+\e{\edge{g^2}{1}{g}} , \e{\edge{g}{2}{1}} ]} + 
                    {[\e{\edge{g^2}{2}{g}} , \e{\edge{g}{2}{1}}]}
           \ = \ {[\e{\edge{g^2}{1}{g}}, \e{\edge{g}{2}{1}}]} \ \neq \ 0.
\end{eqnarray*}
Similarly, we find $b^3\neq 0 $ and $b^4=0$.
So we have relations $a^2=0$ and $b^4=0$.
Turning to the next relation, we have
\begin{align*}
b^2a &= [\e{\edge{g^2}{1}{g}},\e{\edge{g}{2}{1}}] \cdot [\e{\edge{g}{1}{1}}] \\
  &= (\e{\edge{g^2}{1}{g}},\e{\edge{g}{2}{1}},1)(g,g,\e{\edge{g}{1}{1}}) + (\e{\edge{g^2}{1}{g}},g,\e{\edge{g}{2}{1}})
      (g,\e{\edge{g}{1}{1}},1) + (g^2,\e{\edge{g^2}{1}{g}},\e{\edge{g}{2}{1}})(\e{\edge{g}{1}{1}},1,1)\\
  &= [\e{\edge{g^3}{1}{g^2}},\e{\edge{g^2}{2}{g}}+\e{\edge{g^2}{1}{g}},\e{\edge{g}{1}{1}}]+[\e{\edge{g^3}{1}{g^2}},\e{\edge{g^2}{1}{g}},\e{\edge{g}{2}{1}}]
      +[\e{\edge{g^3}{1}{g^2}},\e{\edge{g^2}{1}{g}},\e{\edge{g}{2}{1}}]\\
  &= [\e{\edge{g^3}{1}{g^2}},\e{\edge{g^2}{2}{g}},\e{\edge{g}{1}{1}}]+[\e{\edge{g^3}{1}{g^2}},\e{\edge{g^2}{1}{g}},\e{\edge{g}{1}{1}}].
\end{align*}
{Similarly, we find that}
\[
	 ab^2 = [\e{\edge{g^3}{1}{g^2}},\e{\edge{g^2}{2}{g}},\e{\edge{g}{1}{1}}] 
	\quad\hbox{ and }\quad  
	aba = [\e{\edge{g^3}{1}{g^2}},\e{\edge{g^2}{1}{g}},\e{\edge{g}{1}{1}}],
\]
verifying the relation $b^2a=ab^2+aba$. Analogous calculations yield $abab=baba$.
\end{proof}

Before continuing the proof, we tailor an important result of Takeuchi \cite{Tk} to our needs (see \cite[Lemma 5.2.10]{Mo}). This will be used several times in what follows.

\begin{lemma}[Takeuchi]\label{takeuchi} Fix a bialgebra $B$ with coradical $B_0$. A map $f\in \mathop{Hom}_{\kk}(B,B)$ is convolution invertible if and only if $f\vert_{B_0} \in \mathop{Hom}_{\kk}(B_0,B)$ is convolution invertible.
\end{lemma}

\begin{proof}[Proof of (ii)] Recall briefly the coalgebra structure of $H=T(V) \# \kk G$. Specifically, $\Delta(g)=g\otimes g$, $\Delta(a)=a\otimes 1 + g\otimes a$, $\Delta(b)=b\otimes 1 + g\otimes b$, $a\cdot g = ga$, $b\cdot g = g(a+b)$, $S(a)=-g^{-1}a$, and $S(b)=-g^{-1}b$. 

We need $\Delta(\tilde{I}) \subseteq \tilde{I}\otimes H + H \otimes \tilde{I}$.
The calculation is straightforward, though tedious, so we omit it here. For example,
from $\Delta(b^2) = b^2\otimes 1 + ga\otimes b + g^2\otimes b^2$, we get
\begin{align*}
\Delta(b^4) 
&= b^4 \otimes 1 + g(a^3 + ba^2+aba+b^2a+ab^2)\otimes b + g^2(a^2)\otimes b^2 + g^3(0)\otimes b + g^4 \otimes b^4\\
&\subseteq \tilde{I}\otimes H + H\otimes \tilde{I}.
\end{align*}
%
Once $\tilde{I}$ is known to be a bi-ideal, Takeuchi's result (Lemma \ref{takeuchi}) ensures that it is a Hopf ideal. Indeed, the quotient $H/\tilde{I}$ has coradical $\kk G$, and it is elementary to define the antipode $S$ (the convolution inverse of the identity map) for $\kk G$.
\end{proof}

\begin{proof}[Proof of (iii)] 
It suffices to check the claim for homogeneous skew-primitives (because in graded Hopf algebras, homogeneous components of skew-primitives are skew-primitive). This is again straightforward, though tedious; an explicit argument for degree two follows. Note that
\begin{gather*}
\begin{array}{c@{\,\,\,}c@{\,\,\,}c@{\,\,\,}c@{\,\,\,}c@{\,\,\,}c@{\,\,\,}c@{\,\,\,}c@{\,\,\,}c@{\,\,\,}c@{\,\,\,}c}
\Delta(ba) &=& g^2\otimes ba & + & 
     ba\otimes 1 &+& ga\otimes b &+& gb\otimes a &+& ga\otimes a, \\
\Delta(ab) &=& g^2\otimes ab & + &
     ab\otimes 1 &+& ga\otimes b &+& gb\otimes a, &&  \\
\Delta(b^2) &=& g^2\otimes b^2 &+&
     b^2\otimes 1 &+& ga\otimes b. && && 
\end{array}
\end{gather*}
No linear combination $\lambda ba + \mu ab + \nu b^2$ can be skew-primitive: 
the term $ga\otimes a$ occurs only in $\Delta(ba)$, so we need $\lambda=0$; this, in turn, forces $\mu=0$ and $\nu=0$.
%
\end{proof}

This completes the proof of Theorem \ref{rank2-char2}. An immediate consequence is the following.

\begin{cor}\label{char2-pointed} 
Assume $\kk$ has characteristic 2.
If ${\mathcal B}(V)$ is the Nichols algebra of Theorem \ref{rank2-char2}, then
${\mathcal B}(V)\# \kk G$ is a pointed
Hopf algebra of dimension $16n$ generated by $a$, $b$, and $g$, with relations
\begin{gather*}
	g^n=1, \quad\,\, g^{-1}ag = a, \quad\,\, g^{-1}bg = a + b, \\[0ex]
	a^2=0, \quad\,\, b^4=0, \quad\,\, baba = abab, \quad\,\, b^2a = ab^2 + aba.
\end{gather*}
The coproduct is determined by
$$
\Delta(g) = g\ot g , \quad\,\, \Delta(a) = a\ot 1 + g\ot a, \quad\,\,
  \Delta(b) = b\ot 1 + g\ot b.
$$
\end{cor}

The structure of $\mathcal B(V)$ appears simpler in odd characteristic:
After rescaling generators, $\mathcal B(V)$
is a finite quotient of the Jordanian plane \cite{Sh}, as a consequence
of the following theorem.

\begin{thm}\label{rank2-charp}
Assume $\kk$ has characteristic $p>2$.
The Nichols algebra ${\mathcal B}(V)$ is the quotient of $T(V)$
by the ideal generated by
\begin{gather*}\label{eq:defining relations}
    a^p, \quad\,\, b^{\,p}, \quad\,\,  ba-ab-\frac{1}{2}a^2.
\end{gather*}
It has dimension $p^2$ and basis $\{a^ib^j\mid 0\leq i,j\leq p-1\}$. 
\end{thm}

Our proof will follow the three part strategy used in the proof of Theorem \ref{rank2-char2}. (We again omit the straightforward check of the basis.) 

\begin{proof}[Proof of (i)] To see that the Nichols algebra $\mathcal B(V)$ generated by $a=\e{\edge{g}{1}{1}}$ and $b=\e{\edge{g}{2}{1}}$ has the claimed relations, we appeal to Theorem \ref{product-formula}. For instance, we have
\begin{eqnarray*}
 a^2 = [\e{\edge{g}{1}{1}}]\cdot [\e{\edge{g}{1}{1}}] &=&(\e{\edge{g}{1}{1}},1)(g,\e{\edge{g}{1}{1}})+(g,\e{\edge{g}{1}{1}})
                       (\e{\edge{g}{1}{1}},1)  \\
            &=&[\e{\edge{g^2}{1}{g}},\e{\edge{g}{1}{1}}]+[\e{\edge{g^2}{1}{g}},\e{\edge{g}{1}{1}}]
                 =2[\e{\edge{g^2}{1}{g}},\e{\edge{g}{1}{1}}] , \\
 ab=[\e{\edge{g}{1}{1}}]\cdot [\e{\edge{g}{2}{1}}] &=& (\e{\edge{g}{1}{1}},1)(g,\e{\edge{g}{2}{1}})+
                      (g,\e{\edge{g}{1}{1}})(\e{\edge{g}{2}{1}},1)\\
                    &=& [\e{\edge{g^2}{1}{g}},\e{\edge{g}{2}{1}}]+[\e{\edge{g^2}{2}{g}},\e{\edge{g}{1}{1}}] , \\
 ba=[\e{\edge{g}{2}{1}}]\cdot [\e{\edge{g}{1}{1}}] &=& (\e{\edge{g}{2}{1}},1)(g,\e{\edge{g}{1}{1}})+(g,\e{\edge{g}{2}{1}})
                     (\e{\edge{g}{1}{1}},1)\\
    &=& [\e{\edge{g^2}{2}{g}},\e{\edge{g}{1}{1}}]+[\e{\edge{g^2}{1}{g}},\e{\edge{g}{1}{1}}]+
                  [\e{\edge{g^2}{1}{g}},\e{\edge{g}{2}{1}}].
\end{eqnarray*}
So $ba=ab+\frac{1}{2}a^2$ holds in $\mathcal B(V)$.

Now let $\ell$ be an arbitrary positive integer.
We claim that $a^\ell=\ell![\e{\eedge{1}},\cdots,\e{\eedge{1}}]$, where we have dropped the
subscripts since they are uniquely determined.
This we prove by induction on $\ell$. Clearly $a=[\e{\eedge{1}}]$, and we
have already computed $a^2=2[\e{\eedge{1}},\e{\eedge{1}}]$.
Assume $a^{\ell-1}=(\ell-1)![\e{\eedge{1}},\ldots,\e{\eedge{1}}]$.
Then
\begin{eqnarray*}
 a^\ell &=& (\ell-1)![\e{\eedge{1}}]\cdot [\e{\eedge{1}},\ldots,\e{\eedge{1}}]\\
  &=& (\ell-1)!\left((\e{\eedge{1}},1,\ldots,1)(g^{\ell-1},\e{\eedge{1}},\ldots,\e{\eedge{1}})
                    +(g,\e{\eedge{1}},1,\ldots,1)(\e{\eedge{1}}, g^{\ell-2},\e{\eedge{1}},\ldots,\e{\eedge{1}}) \right.\\
          &&\hspace{1cm}
            \left. +\cdots +(g,\ldots,g,\e{\eedge{1}})(\e{\eedge{1}},\ldots,\e{\eedge{1}},1)\right)\\
    &=& (\ell-1)!\left([\e{\eedge{1}},\ldots,\e{\eedge{1}}]+[\e{\eedge{1}},\ldots,\e{\eedge{1}}] +\cdots +
                     [\e{\eedge{1}},\ldots,\e{\eedge{1}}]\right)\\
     &=& \ell (\ell-1)![\e{\eedge{1}},\ldots,\e{\eedge{1}}] \ =\  \ell![\e{\eedge{1}},\ldots,\e{\eedge{1}}].
\end{eqnarray*}
It follows that $a^p=0$ (and no lower power of $a$ is 0)
since the characteristic of $\kk$ is $p$.

Next we verify that $b^{\,p}=0$. We will not fully compute $b^\ell$, but only
determine some numerical information about the coefficients of paths
in the expansion of $b^\ell$ as an element of the path coalgebra.
To find $b^\ell=([\e{\eedge{2}}])^\ell$,
we first consider all $\ell$-thin splits
of the first factor $[\e{\eedge{2}}]$:
$$(\e{\eedge{2}},1,\ldots,1), \ (g,\e{\eedge{2}},1,\ldots,1), \ldots, 
    \ (g,\ldots,g,\e{\eedge{2}}).
$$
For each of these, we consider all compatible $\ell$-thin splits of the
next factor $[\e{\eedge{2}}]$; for example, corresponding to $(\e{\eedge{2}},1,\ldots
,1)$, we consider
$$
(g,\e{\eedge{2}},1,\ldots,1), \ (g,g,\e{\eedge{2}},1,\ldots,1), \ldots, \
    (g,\ldots,g,\e{\eedge{2}}).
$$
Continuing, we see that $([\e{\eedge{2}}])^\ell$ is a sum of $\ell!$ terms,
each a product of $\ell$ many $\ell$-thin splits of $[\e{\eedge{2}}]$.
Note that
\begin{equation}\label{gigj}
   g^k\cdot \e{\edge{g}{2}{1}}=\e{\edge{g^{k+1}}{2}{g^k}} \ \mbox{ and } \ 
   \e{\edge{g}{2}{1}} \cdot g^i = i\e{\edge{g^{i+1}}{1}{g^i}} + \e{\edge{g^{i+1}}{2}{g^i}}.
\end{equation}
The choice of thin split from which comes the $\e{\edge{g}{2}{1}}$ in the leftmost
position determines the number of factors of $g$ to the left
(say $g^k$) and to the right (say $g^i$) of this $\e{\edge{g}{2}{1}}$.
The corresponding summands in the $\ell$-fold version of \eqref{CR formula} will have $i\e{\edge{g^{\ell-1}}{1}{g^{\ell-2}}}+\e{\edge{g^{\ell-1}}{2}{g^{\ell-2}}}$ in the leftmost position.
Using the same reasoning for the $\e{\edge{g}{2}{1}}$'s in the other positions,
we obtain
\begin{equation}\label{x2n}
  [\e{\eedge{2}}]^\ell=\sum_{0\leq i_2\leq 1, \,\ldots,\,0\leq i_{\ell}\leq \ell -1}
          [i_{\ell} \e{\eedge{1}}+\e{\eedge{2}}, \ldots, i_j\e{\eedge{1}}+\e{\eedge{2}},\ldots,
      i_{2}\e{\eedge{1}}+\e{\eedge{2}}, \e{\eedge{2}}].
\end{equation}
We consider the paths ending in $\e{\eedge{1}}$ and $\e{\eedge{2}}$ separately (the leftmost positions above). 

For the first, we may rewrite the relevant terms of \eqref{x2n} as
\begin{gather*}\label{x2n-1}
 \sum_{0\leq i_\ell\leq \ell-1} i_\ell \sum_{0\leq i_2\leq 1, \,\ldots,\,0\leq i_{\ell-1}\leq
   \ell -2}[\e{\eedge{1}}, \ldots, i_j\e{\eedge{1}}+\e{\eedge{2}},\ldots,
      i_{2}\e{\eedge{1}}+\e{\eedge{2}}, \e{\eedge{2}}],
\end{gather*}
and since the inner sum is independent of $i_\ell$, even as
\[
\binom{\ell}{2} \sum_{0\leq i_2\leq 1, \,\ldots,\,0\leq i_{\ell-1}\leq \ell -2}[\e{\eedge{1}}, \ldots, i_j\e{\eedge{1}}+\e{\eedge{2}},\ldots,
      i_{2}\e{\eedge{1}}+\e{\eedge{2}}, \e{\eedge{2}}].
\]
If $\ell=p$, then $\binom{\ell}{2}$ is divisible by $p$ (i.e., is zero in $\kk$) and these terms contribute zero to the sum \eqref{x2n}. Next, we consider the paths with $\e{\eedge{2}}$ in the leftmost position. The relevant terms of \eqref{x2n} now look like
\begin{gather*}\label{x2n-2}
 \sum_{0\leq i_\ell\leq \ell-1} 1 \sum_{0\leq i_2\leq 1, \,\ldots,\,0\leq i_{\ell-1}\leq
   \ell -2}[\e{\eedge{2}}, \ldots, i_j\e{\eedge{1}}+\e{\eedge{2}},\ldots,
      i_{2}\e{\eedge{1}}+\e{\eedge{2}}, \e{\eedge{2}}],
\end{gather*}
and again the inner sum is independent of $i_\ell$. 
When $\ell=p$, the result upon switching the order of summation is again zero
in $\kk$. We conclude that $b^p$ is zero in $\mathcal B(V)$. Note, moreover, that the coefficient of $[\e{\eedge{2}},\ldots,\e{\eedge{2}}]$ in the
expansion of $b^\ell$ is $\ell!$, so no lower power of $b$ is zero.
\end{proof}

Before turning to part (ii) of the proof, we isolate some useful formulas
holding in the quotient of $H=T(V)\# \kk G$ by 
the ideal generated by $ba-ab-\frac{1}{2}a^2$;
each may be verified with induction. (Below, $\rfac{x}{k} :=  x(x+1)\cdots (x+k-1)$ is the \emph{rising factorial}.) 

\begin{lemma}\label{useful formulas}
The following relations hold in $H/\langle ba-ab-\frac{1}{2}a^2\rangle$ for all $r,\ell \geq 1$:
\begin{align}
\label{eq: b*a}	
   b^r  a^\ell &= \sum_{k=0}^r \binom{r}{k} \frac1{2^k} \rfac{\ell}{k} \,
      a^{\ell+k} b^{r-k}, \\[0.1ex]
\label{eq: b*g}
   b^r  g^\ell &= \sum_{k=0}^r \binom{r}{k} \frac1{2^k} \rfac{2\ell}{k} \,
      g^\ell a^k b^{r-k} ,\\[0.1ex]
\label{eq: Delta(b)}
   \Delta(b^\ell) &= b^\ell \ot 1 
     \ +\ \sum_{k=1}^{\ell-1} \sum_{i=0}^k \binom{\ell}{k}
     \binom{k}{i} \frac{1}{2^i} \rfac{\ell-k}{i} g^{\ell-k} a^i b^{k-i}\ot b^{\ell-k} 
     \ +\ g^\ell\ot b^\ell , \\
\label{eq: S(b)}
   S(b^\ell) &=\sum_{k=0}^{\ell-1} (-1)^{\ell-k} \binom{\ell}{k} \frac{1}{2^k} \rfac{\ell-k}{k} g^{-\ell} a^k b^{\ell-k}.
\end{align}
\end{lemma}

We leave the proof of Lemma \ref{useful formulas} to the reader. Useful in the proofs of \eqref{eq: Delta(b)} and \eqref{eq: S(b)} are the following special cases of \eqref{eq: b*a} and \eqref{eq: b*g}:
\begin{gather*}
   b a^\ell = a^\ell b + \frac{\ell}{2}a^{\ell+1}  \quad\hbox{ and }\quad
   b^r g^{-1} = g^{-1}\bigl(b^r-rab^{r-1} +\frac{1}{4}r(r-1)a^2b^{r-2}\bigr). 
\end{gather*}
Also useful in the proof of Lemma \ref{useful formulas} is the binomial identity enjoyed by rising factorials,
\begin{gather*}\label{eq:pochhamer identity}
   \rfac{x+y}{\ell} = \sum_{k=0}^{\ell} \binom{\ell}{k}\rfac{x}{k}\rfac{y}{\ell-k}.
\end{gather*}

\begin{proof}[Proof of (ii)] The ideal $\tilde{I}=\langle a^p, \ b^p, \
ba-ab-\frac{1}{2}a^2\rangle $ 
of $H$ is a Hopf ideal. To show this, we first check the condition $\Delta(\tilde{I}) \subseteq \tilde{I}\otimes H + H \otimes \tilde{I}$. It is easy to see that 
\begin{gather*} 
   \Delta(a^\ell)=\sum_{k=0}^\ell \binom{\ell}{k}g^{\ell-k} a^k\ot a^{\ell-k} \quad ( \ell\leq p),
\end{gather*}
since $g$ acts trivially on $a$.
In particular, $\Delta(a^p)=a^p\ot 1 + g^p\ot a^p\in \tilde{I}\ot H
+H\ot\tilde{I}$. The generator $b^p$ is handled by Lemma \ref{useful formulas}. We turn to $ba-ab-\frac{1}{2}a^2$:  
\begin{eqnarray*}
 \Delta(ba-ab-\frac{1}{2}a^2) 
   &=& (ba-ab-\frac{1}{2}a^2)\ot 1 +g^2\ot (ba
     -ab-\frac{1}{2}a^2)\\
   &&\hspace{.4cm} + (ga+gb-gb-\frac{1}{2}\cdot 2ga)\ot a
        + (ga-ga)\ot b\\
  &=&  (ba-ab-\frac{1}{2}a^2)\ot 1 +g^2\ot (ba
     -ab-\frac{1}{2}a^2), 
\end{eqnarray*}
which belongs to $\tilde{I}\ot H + H\ot \tilde{I}$. Note that $ba-ab-\frac12 a^2$ is a skew-primitive element of degree~$2$ in $H$. In accordance with Theorem \ref{check-skew-prims}, it will be eliminated in $H/\tilde{I}$. 

We also must argue that $S(\tilde{I}) \subseteq \tilde{I}$. This follows from Takeuchi's result (Lemma \ref{takeuchi}), as in the proof of Theorem \ref{rank2-char2}. (But note that it is quite easy to verify directly, the hard work having been handled by \eqref{eq: S(b)}.)
\end{proof}

\begin{proof}[Proof of (iii)] After (ii), the quotient $H / \tilde{I}$ is a graded Hopf algebra, with degrees $0$ and $1$ spanned by $\kk G$ and $V,$ respectively. Thus to show that $H / \tilde{I}$ has skew-primitive elements only in degrees $0$ and $1$, it suffices to consider homogeneous skew-primitives. 
We work with the basis $\{g^r a^s b^t \mid 0\leq r<n;\, 0\leq s,t< p\}$ of 
$H/\tilde{I}$. 

If $s+t\geq2$, then the element $g^r a^s b^{t}$ is not skew-primitive. To see this, first write
\begin{align}
\Delta(a^s b^t) &=  
\nonumber   \left(\sum_{i=0}^s \binom{s}{i}
      g^{s-i} a^i\ot a^{s-i}\right) \!
   \left( \sum_{j=0}^{t } \sum_{k=0}^j \binom{t }{j}
     \binom{j}{k} \frac{1}{2^k} \rfac{t -j}{k} g^{t -j} a^k b^{\hskip.1ex j-k}\ot b^{t -j} \right) \\
\label{eq: Delta(ab)}
   &=  \sum_{i=0}^s \sum_{j=0}^t \sum_{k=0}^j \binom{s}{i}\binom{t }{j}\binom{j}{k} \frac1{2^k} \rfac{t -j}{k} g^{s+t-i -j} a^{i+k} b^{\hskip.1ex j-k} \otimes a^{s-i} b^{t -j}.
\end{align}
Since $s+t\geq 2$, there will be at least one term in this sum aside from the extremal terms $a^sb^t\otimes 1$ and $g^{s+t}\otimes a^s b^t$; its bi-degree will be strictly between $(s+t,0)$ and $(0,s+t)$. In particular, $a^s b^t$ is not skew-primitive and neither is $g^r a^s b^t$.

We now turn to a linear combination of basis elements $h=\sum_{r,s,t} \alpha_{r,s,t} g^ra^sb^t$. Again, we may assume $s+t$ is constant. If $h\neq 0$, then we may consider the nonzero terms $\alpha_{r,s,t} g^ra^sb^t$ with $t$ largest (say $t_0$). Applying $\Delta$ to the sum, we obtain terms of the form 
\[
	g^{r+s} a^s \otimes g^r b^{t_0}
\]
(set $i=s$ and $j=k=0$ in \eqref{eq: Delta(ab)} and multiply by $g^r \ot g^r$). These are linearly independent in $H/\tilde{I} \ot H/\tilde{I}$; moreover, only basis elements $g^r a^s b^t$ with $t=t_0$ contribute terms of this form to $\Delta(h)$. In order for $h$ to be skew-primitive, we need the corresponding coefficients $\alpha_{r,s,t_0}$ to be zero, violating the choice of $t_0$.
\end{proof}

This completes the proof of Theorem \ref{rank2-charp}. An immediate consequence is the following.

\begin{cor}\label{charp-pointed}
Assume $\kk$ has characteristic $p>2$.
If ${\mathcal B}(V)$ is the Nichols algebra of Theorem \ref{rank2-charp}, then
${\mathcal B}(V)\# \kk G$ is a $p^2n$-dimensional pointed Hopf algebra
generated by $a$, $b$, and $g$, with relations
\begin{gather*}
   g^n=1, \quad\,\, g^{-1}ag = a, \quad\,\, g^{-1}bg = a+b,    \\
   a^p=0,  \quad\,\, b^p  =   0, \quad\,\, ba = ab+\frac{1}{2} a^2.
\end{gather*}
The coproduct is determined by
$$
\Delta(g) = g\ot g , \quad\,\, \Delta(a) = a\ot 1 + g\ot a, \quad\,\,
  \Delta(b) = b\ot 1 + g\ot b.
$$
\end{cor}

\begin{remark}
In this section, we only considered certain two-dimensional indecomposable
Yetter--Drinfeld modules. 
Our preliminary investigations of larger indecomposable
Yetter--Drinfeld modules 
uncovered some relations in the corresponding Nichols algebras but we have not determined whether they are finite dimensional. 
\end{remark}

\section{Lifting}\label{future}

We point out that we found only graded pointed Hopf algebras in Section \ref{charp}. 
The Andruskiewitsch and Schneider classification program \cite{AS}
suggests that, after finding a new Nichols algebra $\mathcal B(V)$, 
one may construct filtered pointed Hopf algebras as
``liftings'' of ${\mathcal B}(V)\# \kk G$, that is those whose associated graded
algebra is ${\mathcal B}(V)\# \kk G$. 
Liftings were found in characteristic 0 in case $G$ is an abelian group by 
Andruskiewitsch and Schneider \cite{AS},
for nonabelian groups in the rank one case by Krop and Radford \cite{KR},
and for the rank one case in positive characteristic by Scherotzke \cite{SSch}. 
We now give some examples of liftings of the Hopf algebras in 
Section \ref{charp}.

In this section, $G=\langle g\rangle \cong \Z/n\Z$, with $n$ divisible by
the characteristic $p$ of $\kk$, and $V$ is the two-dimensional indecomposable
$\kk G$-module described at the beginning of Section~\ref{charp}.

\begin{thm}\label{charp-filtered}
Assume $\kk$ has characteristic $p>2$, and let $\lambda,\mu\in \kk$. 
There is a (filtered) pointed Hopf algebra, whose associated graded Hopf algebra
is ${\mathcal B}(V)\# \kk G$ of Corollary~\ref{charp-pointed}, generated by
$a$, $b$, and $g$,  with relations
\begin{gather*}
   g^n=1, \quad\,\, g^{-1}ag = a, \quad\,\, g^{-1}bg = a+b,    \\
   a^p=\lambda(1-g^p),  \quad\,\, b^p  =   \mu(1-g^p), \quad\,\, ba =ab+\frac{1}{2} a^2 .
\end{gather*}
The coproduct is determined by
$$
\Delta(g) = g\ot g , \quad\,\, \Delta(a) = a\ot 1 + g\ot a, \quad\,\,
  \Delta(b) = b\ot 1 + g\ot b.
$$
\end{thm}

\begin{proof}
Let $\tilde{I}$ be the following ideal of $H=T(V)\# \kk G$:
\[
  \tilde{I} = \langle a^p-\lambda(1-g^p), \ b^p-\mu (1-g^p), \  
   ba-ab-\frac{1}{2}a^2  \rangle .
\]
We must show that $\tilde{I}$ is a Hopf ideal.
It will follow that $H/\tilde{I}$ is a pointed Hopf algebra with associated
graded Hopf algebra as claimed.
It suffices (by Lemma \ref{takeuchi}) to show that 
$\Delta(\tilde{I})\subseteq \tilde{I}\ot H + H\ot \tilde{I}$.
The quadratic relation is unchanged from Theorem \ref{rank2-charp}. This, in turn, means that
the formulas of Lemma \ref{useful formulas} hold in this new algebra. In particular, we have
\begin{eqnarray*}
  \Delta(a^p) &=& a^p\ot 1 + g^p\ot a^p,\\
  \Delta(b^p) &=& b^p\ot 1 + g^p\ot b^p,\\
  \Delta(ba-ab-\frac{1}{2}a^2) &=& (ba-ab-\frac{1}{2}a^2)\ot 1 + g^2\ot (ba-ab-\frac{1}{2}a^2).
\end{eqnarray*}
It follows that 
\begin{eqnarray*}
  \Delta(a^p-\lambda(1-g^p)) & = & a^p\ot 1 + g^p\ot a^p -\lambda(1-g^p)\ot 1
                                         -\lambda g^p\ot (1-g^p)\\
              &=& (a^p -\lambda(1-g^p))\ot 1 + g^p\ot (a^p - \lambda(1-g^p))\\
    & \in & \tilde{I}\ot H + H\ot \tilde{I},
\end{eqnarray*}
and similarly for $\Delta(b^p-\mu(1-g^p))$ and $\Delta(ba-ab-\displaystyle{\frac{1}{2}}a^2)$.
\end{proof}

\begin{remark}
A brief word on the characteristic 2 case. Following the pattern in Theorem \ref{charp-filtered}, it is not obvious how to lift the $a^2$ relation in Corollary \ref{char2-pointed}. However, the presentation in Remark \ref{change-of-basis} yields even more possibilities than found above: changing the quartic relations to read 
\[
  c^4 = \lambda(1-g^4), \ \ d^4 = \mu(1-g^4)\ \hbox{ and }\ c^2d^2 + d^2c^2 = \nu(1-g^4)
\]
gives a three parameter family of liftings. The result is a filtered pointed Hopf algebra generated by $c, d$ and $g$. (The group action is given by $gcg^{-1}=d$ and $gdg^{-1}=c$.)
The proof is similar to that of Theorem \ref{charp-filtered}. 
\end{remark}

We leave for future work the task of determining all liftings of the graded Hopf algebras
in Section \ref{charp}, as well as the study of rank two (graded and filtered) pointed Hopf algebras in positive characteristic $p$
obtained by using more general two-dimensional Yetter--Drinfeld modules.




\def\cprime{$'$}
\providecommand{\bysame}{\leavevmode\hbox to3em{\hrulefill}\thinspace}
\providecommand{\MR}{\relax\ifhmode\unskip\space\fi MR }
\providecommand{\MRhref}[2]{%
  \href{http://www.ams.org/mathscinet-getitem?mr=#1}{#2}
}
\providecommand{\href}[2]{#2}

\end{document}